\documentclass[pdflatex,sn-mathphys-num]{sn-jnl}

\usepackage{graphicx}%
\usepackage{multirow}%
\usepackage{amsmath,amssymb,amsfonts}%
\usepackage{amsthm}%
\usepackage{mathrsfs}%
\usepackage[title]{appendix}%
\usepackage{xcolor}%
\usepackage{textcomp}%
\usepackage{manyfoot}%
\usepackage{booktabs}%
\usepackage{algorithm}%
\usepackage{algorithmic}
\usepackage{listings}%
\usepackage{siunitx}%
\usepackage{tcolorbox}%
\usepackage{caption}%
\usepackage{bm}%
\usepackage{hyperref}

\tcbuselibrary{breakable, skins}

\newtcolorbox[auto counter, number within=section]{breakablealgorithm}[2][]{
    breakable,
    enhanced,
    colback=white,
    colframe=blue!70!black,
    colbacktitle=blue!30!white,
    coltitle=black,
    fonttitle=\bfseries,
    title=Algorithm \thetcbcounter: #2,
    #1
}

\theoremstyle{thmstyleone}%
\newtheorem{theorem}{Theorem}%
\newtheorem{lemma}[theorem]{Lemma}%

\theoremstyle{thmstyletwo}%
\newtheorem{remark}{Remark}%

\theoremstyle{thmstylethree}%
\newtheorem{definition}{Definition}%

\begin{document}

\title[Arnoldi-Enhanced Multivariate Hermite Interpolation of Manifold-Valued Data]
{Arnoldi enhanced multivariate Hermite interpolation of manifold-valued data}

\author[1,2]{\fnm{Yuxuan} \sur{Li}} \email{Yuxuan.Li2502@student.xjtlu.edu.cn}

\author*[1]{\fnm{Qiang} \sur{Niu}}\email{Qiang.Niu@xjtlu.edu.cn}

\author[1]{\fnm{Wubin} \sur{Zhou}} \email{Wubin.Zhou@xjtlu.edu.cn}

\affil[1]{\orgdiv{School of Mathematics and Physics}, \orgname{Xi'an Jiaotong-Liverpool University}, \orgaddress{\street{Suzhou}, \postcode{215123}, \state{Jiangsu}, \country{P. R. China}}}

\affil[2]{\orgdiv{Department of Mathematical Sciences}, \orgname{University of Liverpool}, \orgaddress{\street{Liverpool}, \postcode{L69 7ZX}, \country{United Kingdom}}}

\abstract{This paper integrates the multivariate Arnoldi process into tangent space Hermite interpolation for manifold-valued data. A \(G\)-Arnoldi recurrence constructs a discrete orthogonal polynomial basis, improving numerical conditioning for high-degree approximation. The analysis uses a Markov inequality with a geometry-dependent exponent \(r\); for the convex parameter domains considered here, \(r=2\) is available. Under suitable regularity and sampling assumptions, we establish \(C^0\) and \(C^1\) error bounds and analyze their propagation through the exponential map. When only function values are sampled, these bounds are \(\mathcal O(\sqrt M\,n^{-m})\) and \(\mathcal O(\sqrt M\,n^{2-m})\), respectively. When first derivatives are also sampled, the \(C^1\) bound improves to \(\mathcal O(\sqrt{M(d+1)}\,n^{1-m})\). Here \(n\) is the polynomial degree, \(d\) the parameter dimension, \(m\) the target regularity, and \(M\) the number of sampling points. For sampling sequences satisfying the fill-distance condition with \(M\asymp n^{2d}\), the corresponding bounds become \(\mathcal O(n^{d-m})\), \(\mathcal O(n^{d+2-m})\), and \(\mathcal O(n^{d+1-m})\), yielding superalgebraic convergence for targets admitting smooth extensions. Numerical experiments on \(SO(3)\) and \(S^2\) compare the accuracy and computational cost of the Arnoldi-enhanced method with THI+Kriging, with supplementary experiments demonstrating approximation without derivative samples.}

\maketitle

\section{Introduction}\label{sec1}

In this paper, we consider the multivariate Hermite interpolation of manifold-valued data \cite{Narcowich1995Generalized,JSC2,JSC1}. Let $D \subset \mathbb{R}^d$ be a parameter domain, $\mathcal{M}$ a Riemannian manifold, and consider a  manifold-valued differentiable function
$$
f: D \rightarrow \mathcal{M}, \quad \bm{\omega} \mapsto f(\bm{\omega}).
$$
Consider a sample data set consisting of $k$ parameter locations $\bm{\omega}_1, \ldots, \bm{\omega}_k \in D$ with corresponding function values (manifold locations) and partial derivatives (tangent vectors)
$$
p_j=f\left(\bm{\omega}_j\right) \in \mathcal{M}, v_j^i:=\partial_i f\left(\bm{\omega}_j\right)=\left.\frac{\mathrm{d}}{\mathrm{~d} t}\right|_{t=0} f\left(\bm{\omega}_j+t e_i\right)=\mathrm{d} f\left(\bm{\omega}_j\right)\left[e_i\right] \in T_{f\left(\bm{\omega}_j\right)} \mathcal{M},
$$
where $e_1, \ldots, e_d$ denotes the orthogonal basis of unit vectors in $\mathbb{R}^d$. The Hermite manifold interpolation problem is formalized as follows:
Find a differentiable, manifold-valued function $\hat{f}: D \rightarrow \mathcal{M}$ such that
$$
\begin{aligned}
\hat{f}\left(\bm{\omega}_j\right) & =p_j \in \mathcal{M}, \quad j=1, \ldots, k \\
\partial_i \hat{f}\left(\bm{\omega}_j\right) & =v_j^i \in T_{p_j} \mathcal{M}, \quad j=1, \ldots, k ;\ i=1, \ldots, d.
\end{aligned}
$$
The Hermite interpolation has numerous applications
in robotics \cite{BML,PR95,Vardi2023Geometric}, computer vision \cite{CV}, model reduction \cite{Amsallem2008Interpolation,Amsallem2011Online,Seguin2022Continuation} and image processing \cite{MI,MI2,Wallner2005Convergence,Absil2016Bezier,Vardi2023Geometric}. The tangent space Hermite interpolation (THI) method \cite{JSC1} is a straightforward method to conduct such interpolation, it fixes a base point $q_0\in\mathcal{M}$ and map all data to the tangent space $T_{q_0}\mathcal{M}$ by $w_j=\log _{q_0}\left(p_j\right) \in T_{q_0} \mathcal{M}$ and $\hat{v}_j^i=d\left(\log _{q_0}\right)_{p_j}\left[v_j^i\right] \in T_{q_0} \mathcal{M}$. The base point $q_0$ is typically chosen as the Riemannian barycenter (Karcher mean) of the sample points $\left\{p_j\right\}$. This choice minimizes the sum of squared distances to the samples, keeps the tangent vectors $w_j=\log _{q_0}\left(p_j\right)$ small in norm, and ensures (under suitable sampling density) that all points lie within the injectivity radius. The procedure guarantees that the logarithm map is well-defined. Then the standard Hermite interpolation is performed on the tangent space and finally map the result back to $\mathcal{M}$ by the exponential map.

For multivariate Hermite interpolation problem \cite{MULTIARNOLDI,simultaneous,JSC1}, one seeks a polynomial $p(\boldsymbol{x}) = \sum_{j=1}^g c_j \phi_j(\boldsymbol{x})$ from a chosen basis $\{\phi_j\}$ that minimizes the discrete $\ell_2$-error at a given set of nodes $\mathcal{X} = \{\boldsymbol{x}_i\}_{i=1}^m \subset \mathbb{R}^d$, with the residual possibly including both function values and partial derivatives up to some order. The most natural choice for the basis is the monomials $\boldsymbol{x}^{\boldsymbol{\alpha}}$, ordered by total degree, which leads to the classical multivariate Vandermonde matrix whose columns are the evaluations of the monomials at $\mathcal{X}$, and when derivatives are present, the resulting discrete data becomes a confluent Vandermonde matrix whose entries also contain the evaluations of the derivatives of the monomials; the coefficient vector $\boldsymbol{c}$ can be obtained by the solution of $V\boldsymbol{c} \approx \boldsymbol{f}$, where $\boldsymbol{f}$ contains the observed function values and derivative data. 

A widely recognized difficulty is the ill-conditioning of Vandermonde matrices \cite{Golub1996Matrix,Ciaramella2026Gentle}. Even for univariate problems, the condition number of $V$ grows exponentially with the polynomial degree. In the multivariate setting, near-linear dependence of monomials renders the standard basis numerically unstable. To circumvent the difficulty, Brubeck, Nakatsukasa, and Trefethen rediscovered the Vandermonde with Arnoldi (V+A) method \cite{SIAM}, which avoids explicit construction of $V$ and instead employs the Arnoldi process to generate a new basis of discrete orthogonal polynomials whose evaluation matrix $Q$ has orthogonal columns and spans the same space as $V$ \cite{GraggReichel1987,GustafssonPS2009,Stylianopoulos2013}. V+A transforms the least-squares problem into $Q\boldsymbol{c} \approx \boldsymbol{f}$ with a perfectly conditioned coefficient matrix, while the Arnoldi process produces an upper triangular matrix $\widetilde{R}$ encoding a recurrence for efficient evaluation of the polynomial and its derivatives at new points. 

The success of V+A inspired important extensions, for example \cite{confluent, Faghih2025,VanBuggenhout2023} and \cite{MULTIARNOLDI}. The univariate confluent V+A \cite{confluent} incorporates derivative information by constructing a Krylov subspace that captures the structure of the confluent Vandermonde matrix, enabling stable Hermite interpolation using both function values and derivatives. The paper \cite{MULTIARNOLDI} further unifies these ideas into the multivariate confluent V+A with $G$-Arnoldi (MV+G-A), which handles multivariate function values and higher-order partial derivatives within a single least-squares formulation. The central innovation is the introduction of an application-dependent $G$-inner product $\langle \boldsymbol{x}, \boldsymbol{y} \rangle_G = \boldsymbol{x}^{\mathrm{T}} G \boldsymbol{y}$ into the Gram–Schmidt step of the Arnoldi process, where the positive semi-definite matrix $G$ is chosen according to the available data: for example, on a 2-dimensional domain, $m_0$ interior nodes provide only function values and $m_1$ boundary nodes provide both function values and normal derivatives, one constructs a selection matrix $L$ that extracts exactly the observed quantities from the full evaluation matrix (which contains function values and all partial derivatives at all nodes), and then sets $G = L^{\mathrm{T}} L$, where 
\[
L = \left[\begin{array}{c|cc|cc}
I_m & & & & \\
\hline & 0_{m_1 \times m_0} & I_{m_1} & & \\
\hline & & & 0_{m_1 \times m_0} & I_{m_1}
\end{array}\right] \in \mathbb{R}^{(m+2m_1) \times 3m}.
\]
The Arnoldi process then performs a node- and derivative-specific orthogonalization, and generate a basis orthogonal with respect to the $G$-inner product and a well conditioned least-squares coefficient matrix $A = L Q$, the resulting linear system $A \boldsymbol{c} = \boldsymbol{b}$ is therefore perfectly conditioned and can be solved by a single matrix‑vector multiplication $\boldsymbol{c} = A^{\mathsf{T}} \boldsymbol{b}$. And the resulting polynomial is represented in a basis of discrete $G$-orthogonal polynomials satisfying an explicit recurrence encoded in $\widetilde{R}$, allowing function values and prescribed partial derivatives at new nodes to be computed efficiently without matrix inversion. The flexibility of MV+G-A is demonstrated on multivariate Hermite least-squares problems with scattered derivative data and on solving linear PDEs with various boundary conditions on irregular domains, where in all cases the coefficient matrix remains perfectly conditioned and the procedure involves only matrix-vector multiplications. In summary, the multivariate confluent Vandermonde with $G$-Arnoldi provides a numerically robust, computationally efficient and flexible tool for multivariate Hermite approximation.

In this paper, we revisit the THI method \cite{JSC1} in Section~\ref{sec2}, show that it can be reformulated as a Hermite interpolation problem and the original approach relies on Gradient-enhanced kriging \cite{kriging1}, we adopt the multivariate Arnoldi method \cite{MULTIARNOLDI} instead. Then in Section~\ref{sec:multivariate_arnoldi} we briefly review this algorithm. In Section~\ref{sec4}, we derive the $C^0$ error bound for the multivariate Arnoldi approximation, demonstrating that the $C^0$ error decays algebraically with the total degree of the polynomial basis but increases with the total number of sampling points, and the total number of sampling points depends on the total degree of the polynomial basis. The estimation for $C^1$ error is also given in Section~\ref{sec43}, we show that whether or not derivative sample data is provided, the derivatives can still be approximated, only the rate of convergence differs. These provide a theoretical basis for why the multivariate Arnoldi method generates better fitting results than Kriging in the following numerical examples. Section~\ref{sec5} simply recalls error propagation estimates on manifolds. Finally, Section~\ref{sec6} presents the numerical experiments to illustrate the performance of the proposed approach.

\section{Tangent space Hermite interpolation(THI) method}
\label{sec2}
 
The THI method was first proposed in \cite{JSC1}. Instead of working directly on the curved manifold $\mathcal{M}$, we fix a base point $q_0\in\mathcal{M}$ which is chosen as the Riemannian barycenter (Karcher mean) of the sample points $\left\{p_j\right\}$ and map all data to the tangent space $T_{q_0}\mathcal{M}$, which is a vector space. Then the standard Hermite interpolation can be performed on the tangent space and finally map the result back to $\mathcal{M}$ via the exponential map.

Given sample points $\bm{\omega}_j\in\mathbb{R}^d$, function values $p_j = f(\bm{\omega}_j)\in\mathcal{M}$, and partial derivatives $v_j^i = \partial_i f(\bm{\omega}_j)\in T_{p_j}\mathcal{M}$, then \cite{Absil2008Optimization,Fletcher2013Geodesic}:

- Points are mapped to the tangent space:  
  $w_j = \log_{q_0}(p_j) \in T_{q_0}\mathcal{M}$. To avoid confusion, one should distinguish $w$ and $\bm{\omega}$.

- Derivatives are transported via the differential of the logarithm map:  
  $\hat{v}_j^i = d(\log_{q_0})_{p_j}[v_j^i] \in T_{q_0}\mathcal{M}$.  
  
  In practice it can be approximated by finite differences:  
  $$\hat{v}_j^i \approx \frac{\log_{q_0}\bigl(\operatorname{Exp}_{p_j}(\Delta t\,v_j^i)\bigr) - \log_{q_0}\bigl(\operatorname{Exp}_{p_j}(-\Delta t\,v_j^i)\bigr)}{2\Delta t}.$$

\subsection{Interpolation in the tangent space} \label{kriging}
We construct an interpolant $\hat{f}_{\text{tan}} :\mathbb{R}^d\to T_{q_0}\mathcal{M}$ satisfying  
$\hat{f}_{\text{tan}}(\bm{\omega}_j)=w_j$ and $\partial_i\hat{f}_{\text{tan}}(\bm{\omega}_j)=\hat{v}_j^i$.  
Using scalar weight functions, we write  
$$\hat{f}_{\text{tan}}(\bm{\omega}) = \sum_{j=1}^k \phi_j(\bm{\omega}) w_j + \sum_{i=1}^d\sum_{l=1}^k \psi_{i,l}(\bm{\omega})\,\hat{v}_l^i,$$
where the $\phi_j$ and $\psi_{i,l}$ are scalar functions on $\mathbb{R}^d$ that satisfy the Hermite conditions  
$$\phi_j(\bm{\omega}_l)=\delta_{jl},\quad \partial_i\phi_j(\bm{\omega}_l)=0,\qquad
\psi_{i,l}(\bm{\omega}_m)=0,\quad \partial_{i'}\psi_{i,l}(\bm{\omega}_m)=\delta_{ii'}\delta_{lm}.$$
If we write $\vec{\varphi}_j:=\left(\varphi_j\left(\bm{\omega}_1\right), \ldots, \varphi_j\left(\bm{\omega}_k\right)\right)$ and $\vec{\psi}_{i, l}:=\left(\psi_{i, l}\left(\bm{\omega}_1\right), \ldots, \psi_{i, l}\left(\bm{\omega}_k\right)\right)$. Then the Hermite condition can be written in the form of $$
\left(\begin{array}{c}
\vec{\varphi}_j \\
\partial_1 \vec{\varphi}_j \\
\vdots \\
\partial_d \vec{\varphi}_j
\end{array}\right)=\left(\begin{array}{c}
e_j \\
\mathbf{0} \\
\vdots \\
\mathbf{0}
\end{array}\right), \quad\left(\begin{array}{c}
\vec{\psi}_{1, l} \\
\partial_1 \vec{\psi}_{1, l} \\
\vdots \\
\partial_d \vec{\psi}_{1, l}
\end{array}\right)=\left(\begin{array}{c}
\mathbf{0} \\
e_l \\
\vdots \\
\mathbf{0}
\end{array}\right), \ldots,\left(\begin{array}{c}
\vec{\psi}_{d, l} \\
\partial_1 \vec{\psi}_{d, l} \\
\vdots \\
\partial_d \vec{\psi}_{d, l}
\end{array}\right)=\left(\begin{array}{c}
\mathbf{0} \\
\mathbf{0} \\
\vdots \\
e_l
\end{array}\right) \in \mathbb{R}^{k(d+1)} .
$$

In \cite{JSC1}, the weight functions are constructed by Gradient-enhanced Kriging method \cite{kriging1}, which assumes that the unknown function is a realization of a Gaussian process with a prescribed correlation kernel. The kernel defines the covariance between any two function values, as well as between function values and derivatives, and between derivatives themselves. Using the kernel, a linear system is built from the observed data (values and derivatives). The solution gives a predictor that is a linear combination of the data, with coefficients determined by the correlation structure. 

In this paper, we propose Arnoldi‑enhanced method, which follows the same linearisation strategy as THI: a base point $q_0\in\mathcal{M}$ is fixed, and all manifold data are pulled back to the tangent space $T_{q_0}\mathcal{M}\cong\mathbb{R}^m$ via $w_j=\operatorname{Log}_{q_0}(p_j)$ and $\hat{v}_j^i=d\left(\log _{q_0}\right)_{p_j}\left[v_j^i\right] \in T_{q_0} \mathcal{M}$. The resulting $\mathbf{w}:\mathbb{R}^d\to\mathbb{R}^m$ is then directly approximated componentwise: for each coordinate $c=1,\dots,m$ we construct a multivariate polynomial $P^{c}(\bm{\omega})$ to approximate, so $\mathbf{w}$ could be approximated by $\mathbf{\widehat{w}}((\bm{\omega}))=(P^{1}(\bm{\omega}),\cdots,P^{m}(\bm{\omega}))$. By multivariate Arnoldi, $P^{c}(x)$ could be written as $\sum_{j=1}^g c_j^{[c]}\,\xi_j(\bm{\omega})$ where the basis $\{\xi_j\}_{j=1}^g$ can be shared by all components. At a new point $\bm{\omega}^*$, the basis values can be accurately generated from multivariate Arnoldi process, and the final interpolant on $\mathcal{M}$ can be reconstructed as $\widehat{f}(\bm{\omega}^*) = \operatorname{Exp}_{q_0}\!\bigl(\widehat{\mathbf{w}}(\bm{\omega}^*)\bigr)$.

\section{Multivariate Arnoldi process for polynomial approximation}
\label{sec:multivariate_arnoldi}

In this section, we summarize the construction and evaluation
of the discrete orthogonal polynomial basis used in our
method, following \cite{MULTIARNOLDI}.
Let $\mathcal{X}=\{x_j\}_{j=1}^{M}\subset\mathbb{R}^d$,
and let $\{\varphi_i\}_{i=1}^{N}$ be the monomial basis of
$\mathcal{P}_{d,n}$ in graded reverse lexicographic
(grevlex) order, where
$N=\binom{n+d}{d}$ and $\varphi_1=1$.

For a polynomial $p$, define its sampling vector by
\[
\mathcal{S}p=
\begin{cases}
\bigl(p(x_j)\bigr)_{j=1}^{M},
& \text{for function-value sampling},\\[2mm]
\bigl(
(p(x_j))_{j=1}^{M},
(\partial_1p(x_j))_{j=1}^{M},
\ldots,
(\partial_dp(x_j))_{j=1}^{M}
\bigr)^\top,
& \text{1st derivatives are also sampled}.
\end{cases}
\]
Here all sampling vectors are understood as column vectors.
The corresponding Vandermonde or confluent Vandermonde
matrix is
\[
V=[\mathcal{S}\varphi_1,\ldots,\mathcal{S}\varphi_N].
\]
Given the corresponding data vector $b$, the unweighted
least-squares approximation minimizes
$\|\mathcal{S}p-b\|_2$ over $p\in\mathcal{P}_{d,n}$.
The purpose of the Arnoldi construction is to solve this
problem in a discrete orthogonal basis without explicitly
forming the ill-conditioned matrix $V$.

\subsection{Construction of the discrete orthogonal basis}

For each $i\ge2$, choose the smallest index $s_i<i$ for
which
\[
\varphi_i(x)=x_{u_i}\varphi_{s_i}(x),
\qquad u_i\in\{1,\ldots,d\}.
\]
Coordinate multiplication induces an operator $X_{u_i}$
on sampling vectors satisfying
\[
\mathcal{S}(x_{u_i}p)=X_{u_i}\mathcal{S}p.
\]
For function values, this is simply multiplication by
the sampled coordinate. When derivatives are included,
the operator also incorporates the product rule
\[
\partial_\ell(x_{u_i}p)
=
\delta_{\ell,u_i}p+x_{u_i}\partial_\ell p,
\qquad \ell=1,\ldots,d.
\]

Let $G$ be the symmetric positive semidefinite matrix
specifying the discrete bilinear form
\[
\langle y,z\rangle_G=y^\top Gz,
\qquad
\|y\|_G=(y^\top Gy)^{1/2}.
\]
We assume that $V^\top GV$ is positive definite, so that
this form is an inner product on the sampled polynomial
space. Starting from
\[
v_1=\mathcal{S}1,\qquad
\widetilde R_{1,1}=\|v_1\|_G,\qquad
q_1=\frac{v_1}{\widetilde R_{1,1}},
\]
the Arnoldi process forms $k_i=X_{u_i}q_{s_i}$ and
orthogonalizes it against the preceding basis vectors:
\[
\widetilde R_{j,i}=\langle q_j,k_i\rangle_G,
\qquad j=1,\ldots,i-1,
\]
\[
z_i=k_i-\sum_{j=1}^{i-1}\widetilde R_{j,i}q_j,
\qquad
\widetilde R_{i,i}=\|z_i\|_G,
\qquad
q_i=\frac{z_i}{\widetilde R_{i,i}}.
\]
These formulas describe the construction in exact
arithmetic; reorthogonalization is used in the numerical
implementation. The grevlex ordering ensures that the
resulting vectors span the same sampled polynomial
space as the monomial columns; see \cite{MULTIARNOLDI}.
Thus, writing
\[
Q=[q_1,\ldots,q_N],
\]
we have $Q^\top GQ=I_N$.

\subsection{Least-squares fitting and evaluation}

Each column $q_i$ is the sampling vector of a polynomial
$\xi_i\in\mathcal{P}_{d,n}$. These polynomials form a
discrete orthogonal basis, and the approximation is
represented as
\[
p(x)=\sum_{i=1}^{N}c_i\xi_i(x).
\]
More generally, a data-selection or weighting matrix $L$
defines $G=L^\top L$. For the correspondingly selected
or weighted data vector $b$, the coefficient problem is
\[
\min_c\|LQc-b\|_2.
\]
Since $(LQ)^\top(LQ)=I_N$, its solution is
\[
c=(LQ)^\top b.
\]
The error analysis in the next section concerns
unweighted least squares, with $G=L=I$, for which
$c=Q^\top b$. Under full column rank, this defines a
linear approximation operator that reproduces every
polynomial in $\mathcal{P}_{d,n}$.

The coefficients stored in $\widetilde R$ also determine
the basis at new points:
\[
\xi_1(x)=\frac{1}{\widetilde R_{1,1}},
\qquad
\xi_i(x)=
\frac{x_{u_i}\xi_{s_i}(x)
-\sum_{j=1}^{i-1}\widetilde R_{j,i}\xi_j(x)}
{\widetilde R_{i,i}},
\quad i\ge2.
\]
Differentiating this recurrence gives
\[
\partial_\ell\xi_i(x)=
\frac{
\delta_{\ell,u_i}\xi_{s_i}(x)
+x_{u_i}\partial_\ell\xi_{s_i}(x)
-\sum_{j=1}^{i-1}\widetilde R_{j,i}
 \partial_\ell\xi_j(x)}
{\widetilde R_{i,i}},
\]
with $\partial_\ell\xi_1=0$. Consequently, both the
polynomial and its derivatives can be evaluated without
conversion to monomial coefficients.

For manifold-valued data, the same basis is used to
approximate all components of the tangent-space data.
The resulting polynomial approximation $\widehat w$
fits these data in the least-squares sense and need not
satisfy every interpolation condition exactly.
Finally, the manifold-valued approximation is obtained by
\[
\widehat f(x)
=
\operatorname{Exp}_{q_0}\bigl(\widehat w(x)\bigr).
\]

\section{The error analysis of multivariate Arnoldi}
\label{sec4}
In this section, we derive the $C^0, C^1$ error bound for the multivariate Arnoldi approximation, demonstrating that the $C^0$ error decays algebraically with the total degree of the polynomial basis but increases with the total number of sampling points, and the total number of sampling points depends on the total degree of the polynomial basis. The estimation for $C^1$ error is also given, we show that it costs more total order of the polynomial than the $C^0$ case. Besides, in Section~\ref{sec43} we will show that whether or not derivative sample data are provided, the derivatives can still be approximated, only the rate of convergence differs. Throughout this section, we assume the function we want to approximate is in $C^m$ with $m$ large enough, actually in our experiment, we choose the smooth function.

We first discretize the domain in $\mathbb{R}^d$ with grid size $h$ defined by $\delta(\bar{\Omega}, A)=\sup _{x \in \Omega} \operatorname{dist}(x, A)$ where $A$ is the sample points set, and then we could get the polynomial approximation by solving the following least square problem
by multivariate Arnoldi
$$
\sum_{i=1}^m|v\left(x_i\right)-p(x_i)|^2+\sum_{1 \leq|\alpha| \leq k} \sum_{i=1}^m\left|D^\alpha( v(x_i)-p(x_i))\right|^2.
$$
Besides, recall the definition of $G$ and $L$, if we have all the data we need on every sampling points, we could choose $G=L=I$ in Section~\ref{sec:multivariate_arnoldi}. Since the $G$ and $L$ doesn't affect the core idea of our error analysis, we assume $G=L=I$ without loss of generality.
 
To conduct the error analysis of multivariate Arnoldi, we will first introduce the Markov inequality, then the simultaneous approximation inequality (\ref{simulinequal}) in \cite{simultaneous} showing that there exists a multivariate polynomial of total degree at most $n$ that can simultaneously approximate a $C^m$ function and its derivatives up to order $\min (n, m)$.

\begin{definition}[Markov inequality]
\label{def:markov}
A compact set $K\subset\mathbb{R}^d$ is said to admit
a Markov inequality with exponent $r>0$ and constant
$B_K>0$ if
\[
\sup_{x\in K}\|\nabla p(x)\|_2
\leq B_K n^r\|p\|_{\infty,K},
\qquad
p\in\mathcal{P}_{d,n},\quad n\geq 1,
\]
where $\|p\|_{\infty,K}:=\max_{x\in K}|p(x)|$.
The constants $r$ and $B_K$ are independent of $n$
and $p$.
\end{definition}

Every convex body $K\subset\mathbb{R}^d$ admits such
an inequality with
$$
r=2,\qquad B_K=\frac{4}{w(K)},
$$
where $w(K)$ denotes the minimal width of $K$.
For centrally symmetric convex bodies, one may take
$B_{K}=2/w(K)$; see \cite{Delfour2011Shapes},
Section~6.4, related work are also summarized in \cite{OL}.

\begin{lemma}
\label{lem:norming}
Let $K\subset\mathbb{R}^d$ be a compact convex set
with nonempty interior. Then $K$ satisfies
the Markov inequality in Definition~\ref{def:markov}
with fixed constants $r=2$ and $B_K>0$, namely,
\[
\sup_{x\in K}\|\nabla p(x)\|_2
\leq B_K n^2\|p\|_{\infty,K},
\qquad p\in\mathcal{P}_{d,n},\quad n\geq 1.
\]
Set $c_K=(2B_K)^{-1}$.
Then, for every integer $n\geq 1$ and every nonempty
finite set $A\subset K$ with fill distance
\[
h_{A,K}:=\sup_{x\in K}\min_{a\in A}\|x-a\|_2
\leq c_K n^{-2},
\]
we have
\[
\|v\|_{\infty,K}
\leq 2\max_{a\in A}|v(a)|,
\qquad v\in\mathcal{P}_{d,n}.
\]
\end{lemma}
\begin{proof}
By the Markov inequality, there exists a constant $B_K>0$, independent of $n$, such that
$$
\sup _{x \in K}\|\nabla v(x)\|_2 \leq B_K n^2\|v\|_{\infty, K}, \quad v \in \mathcal{P}_{d, n}.
$$
Set $c_K=\left(2 B_K\right)^{-1}$. For any $x \in K$, choose $a_x \in A$ satisfying
$$
\left\|x-a_x\right\|_2=\min _{a \in A}\|x-a\|_2 \leq h_{A, K} .
$$
Since $K$ is convex, the segment joining $a_x$ to $x$ lies in $K$. The fundamental theorem of calculus and the Cauchy-Schwarz inequality therefore give
$$
\begin{aligned}
\left|v(x)-v\left(a_x\right)\right| & =\left|\int_0^1 \nabla v\left(a_x+t\left(x-a_x\right)\right) \cdot\left(x-a_x\right) d t\right| \\
& \leq h_{A, K} \sup _{y \in K}\|\nabla v(y)\|_2 \\
& \leq B_K n^2 h_{A, K}\|v\|_{\infty, K} .
\end{aligned}
$$
Taking the supremum over $x \in K$, we obtain
$$
\|v\|_{\infty, K} \leq \max _{a \in A}|v(a)|+B_K n^2 h_{A, K}\|v\|_{\infty, K} .
$$
Since $h_{A, K} \leq c_K n^{-2}$, we have $B_K n^2 h_{A, K} \leq \frac{1}{2}$. Rearranging yields
$$
\frac{1}{2}\|v\|_{\infty, K} \leq \max _{a \in A}|v(a)|,
$$
which proves the assertion.
\end{proof}

Then we introduce a famous multivariate simultaneous approximation theorem showing that there exists a multivariate polynomial of total degree at most $n$ can simultaneously approximate a $C^m$ function and its derivatives up to order $\min (n, m)$.
\begin{theorem} (\textbf{Theorem~1 in \cite{simultaneous}})\label{simul} Let $f$ be a function of compact support on $\mathbf{R}^N$, of class $C^m$ where $0 \leq m<\infty$, and let $K$ be a compact subset of $\mathbf{R}^N$ which contains the support of $f$. Then for each nonnegative integer $n$ there is a polynomial $p_n$ of degree at most $n$ on $\mathbf{R}^N$ with the following property: for each multi-index $\alpha$ with $|\alpha| \leq \min \{m, n\}$ we have
\begin{equation}\label{simulinequal}
\sup _K\left|D^\alpha\left(f-p_n\right)\right| \leq \frac{C}{n^{m-|\alpha|}} \omega_{f, m}\left(\frac{1}{n}\right),
\end{equation}
where $C$ is a positive constant depending only on $N, m$, and the diameter of $K$ and $\omega_{f, m}(\delta)=\sup _{|\gamma|=m}\left(\sup _{|x-y| \leq \delta}\left|D^\gamma f(x)-D^\gamma f(y)\right|\right)$.
\end{theorem}
We impose the \textbf{extension assumption}, which means that
there exist an open neighborhood $U$ of $K$ and a
function $F\in C^m(U)$ such that $F|_K=f$.
Multiplying $F$ by a smooth cutoff equal to one near
$K$ and extending the product by zero yields a
compactly supported $C^m$ extension to $\mathbb{R}^d$,
as required by Theorem~\ref{simul}.

With the help of above lemmas, we could prove the theorem we need in the error analysis of multivariate Arnoldi. 

\subsection{$C^{0}$ error analysis.}
\label{sec42}
For $C^{0}$ error analysis, we only need the sampling data of the function values, the sampling data of the derivative values are not required.

\begin{theorem}\label{thm:C0}
Let $\Omega\subset\mathbb{R}^d$ be a bounded convex domain
and set $K=\overline{\Omega}$. Then $K$ admits a
Markov inequality with a fixed exponent $r=2$, and let
$c_K>0$ be the constant in Lemma~\ref{lem:norming}.
Let $m\ge1$ be an integer, and suppose that $f$ satisfies the extension assumption.

For an integer $n\ge1$, let
$A=\{a_1,\ldots,a_M\}\subset K$ be a set of distinct
sampling points satisfying
\[
h_{A,K}:=\sup_{x\in K}\min_{a\in A}\|x-a\|_2
\le c_K n^{-2}.
\]
Let $\mathcal{L}_n f\in\mathcal{P}_{d,n}$ be the discrete
least-squares approximation based only on function values,
computed using the Arnoldi basis:
\[
\mathcal{L}_n f
=
\operatorname*{arg\,min}_{p\in\mathcal{P}_{d,n}}
\sum_{j=1}^{M}|f(a_j)-p(a_j)|^2.
\]
Then there exists a constant $C>0$, independent of $n$,
$M$, and the sampling set $A$, such that
\[
\|f-\mathcal{L}_n f\|_{\infty,K}
\le C\sqrt{M}\,n^{-m}.
\]
The constant may depend on $f$, $m$, $d$, and $K$.
\end{theorem}

\begin{proof}
By Lemma~\ref{lem:norming},
\[
\|v\|_{\infty,K}
\le 2\max_{a\in A}|v(a)|,
\qquad v\in\mathcal{P}_{d,n}.
\]
Thus a polynomial in $\mathcal{P}_{d,n}$ that vanishes
on $A$ must vanish on $K$, and hence identically,
since $K$ has nonempty interior.
Consequently, the polynomial evaluation matrix has
full column rank, and the discrete least-squares
operator $\mathcal{L}_n$ is well defined, linear,
and reproduces polynomials:
\[
\mathcal{L}_n p=p,
\qquad p\in\mathcal{P}_{d,n}.
\]

Under the extension assumption,
Theorem~\ref{simul} provides a polynomial
$p_n^*\in\mathcal{P}_{d,n}$ satisfying
\[
\|f-p_n^*\|_{\infty,K}\le C_f n^{-m},
\]
where $C_f$ is independent of $n$.
Set $g=f-p_n^*$. By linearity and polynomial reproduction,
\[
f-\mathcal{L}_n f=g-\mathcal{L}_n g.
\]
Therefore,
\begin{equation}\label{error1}
\|f-\mathcal{L}_n f\|_{\infty,K}
\le
\|g\|_{\infty,K}
+
\|\mathcal{L}_n g\|_{\infty,K}.
\end{equation}

Let $N=\dim\mathcal{P}_{d,n}$, and let
$Q\in\mathbb{R}^{M\times N}$ be the evaluation matrix
of the discrete orthonormal Arnoldi basis, so that
$Q^\top Q=I_N$. For the sampling vector
\[
\widetilde g
=
\bigl(g(a_1),\ldots,g(a_M)\bigr)^\top,
\]
the fitted values satisfy
\[
(\mathcal{L}_n g)|_A=QQ^\top\widetilde g.
\]
Applying Lemma~\ref{lem:norming} to
$\mathcal{L}_n g\in\mathcal{P}_{d,n}$ gives
\begin{equation}\label{error2}
\begin{aligned}
\|\mathcal{L}_n g\|_{\infty,K}
&\le
2\|(\mathcal{L}_n g)|_A\|_{\ell^\infty}\\
&=
2\|QQ^\top\widetilde g\|_{\ell^\infty}\\
&\le
2\|QQ^\top\|_\infty
\|\widetilde g\|_{\ell^\infty},
\end{aligned}
\end{equation}
where $\|\cdot\|_\infty$ denotes the induced
maximum-row-sum matrix norm.

Set $P=QQ^\top\in\mathbb{R}^{M\times M}$.
By the Cauchy--Schwarz inequality,
\[
\begin{aligned}
\|P\|_\infty
&=
\max_{1\le i\le M}\sum_{j=1}^{M}|P_{ij}|\\
&\le
\sqrt{M}\max_{1\le i\le M}
\left(\sum_{j=1}^{M}|P_{ij}|^2\right)^{1/2}\\
&=
\sqrt{M}\max_{1\le i\le M}\|P^\top e_i\|_2\\
&\le
\sqrt{M}\,\|P^\top\|_2
=
\sqrt{M}\,\|P\|_2,
\end{aligned}
\]
where $e_i$ is the $i$th coordinate vector and
$\|\cdot\|_2$ denotes the induced Euclidean matrix norm.
Since $Q^\top Q=I_N$, the matrix $P$ is a nonzero
orthogonal projection, and hence $\|P\|_2=1$.
Therefore,
\[
\|QQ^\top\|_\infty
\le
\sqrt{M}\,\|QQ^\top\|_2
=
\sqrt{M}.
\]

Moreover,
\[
\|\widetilde g\|_{\ell^\infty}
=
\max_{a\in A}|g(a)|
\le
\|g\|_{\infty,K}.
\]
Substituting these estimates into \eqref{error2},
we obtain
\[
\|\mathcal{L}_n g\|_{\infty,K}
\le 2\sqrt{M}\,\|g\|_{\infty,K}.
\]
Combining this with \eqref{error1} yields
\[
\begin{aligned}
\|f-\mathcal{L}_n f\|_{\infty,K}
&\le
(1+2\sqrt{M})\|f-p_n^*\|_{\infty,K}\\
&\le
C_f(1+2\sqrt{M})n^{-m}\\
&\le
3C_f\sqrt{M}\,n^{-m},
\end{aligned}
\]
since $M\ge1$. This proves the asserted estimate,
with a constant independent of $n$, $M$, and the
sampling set.
\end{proof}
\subsection{$C^{1}$ error analysis.}
\label{sec43}
\subsubsection{Achieve $C^{1}$ estimate given the sampling data of the derivative values.}
\begin{theorem}\label{thm:C1-Hermite}
Let $\Omega\subset\mathbb{R}^d$ be a bounded convex domain
and set $K=\overline{\Omega}$. Then $K$ admits
a Markov inequality with a fixed exponent $r=2$, and
let $c_K>0$ be the constant in Lemma~\ref{lem:norming}.
Let $m\ge1$ be an integer, and suppose that $f$ satisfies the extension assumption.

For an integer $n\ge1$, let
$A=\{a_1,\ldots,a_M\}\subset K$ be a set of distinct
sampling points satisfying
\[
h_{A,K}:=\sup_{x\in K}\min_{a\in A}\|x-a\|_2
\le c_K n^{-2}.
\]
Let $\mathcal{H}_n f\in\mathcal{P}_{d,n}$ be the
Hermite least-squares approximation computed using
the Arnoldi basis:
\[
\mathcal{L} (f)
=
\operatorname*{arg\,min}_{p\in\mathcal{P}_{d,n}}
\left(
\sum_{j=1}^{M}|f(a_j)-p(a_j)|^2
+
\sum_{i=1}^{d}\sum_{j=1}^{M}
|\partial_i f(a_j)-\partial_i p(a_j)|^2
\right).
\]
Then there exists a constant $C>0$, independent of
$n$, $M$, and the sampling set $A$, such that
\begin{equation}\label{finalerror2}
\|\nabla f-\nabla\mathcal{L}(f)\|_{\infty,K}
\le C\sqrt{M(d+1)}\,n^{1-m},
\end{equation}
where
\[
\|\nabla u\|_{\infty,K}
:=
\sup_{x\in K}\|\nabla u(x)\|_2.
\]
The constant may depend on $f$, $m$, $d$, and $K$.
\end{theorem}

\begin{proof}
We use the first-order seminorm
\[
|u|_{1,\infty,K}
:=
\max_{1\le i\le d}\|\partial_i u\|_{\infty,K}.
\]
By Lemma~\ref{lem:norming}, the function-value
evaluation matrix on $\mathcal{P}_{d,n}$ has full
column rank. Hence the augmented evaluation matrix
including all first-order partial derivatives also
has full column rank. Consequently, the Hermite
least-squares operator $\mathcal{L}$ is well defined,
linear, and satisfies
\[
\mathcal{L}(p)=p,
\qquad p\in\mathcal{P}_{d,n}.
\]

Under the extension assumption,
Theorem~\ref{simul} provides a polynomial
$p^*\in\mathcal{P}_{d,n}$ such that
\[
\|f-p^*\|_{\infty,K}\le C_f n^{-m},
\qquad
|f-p^*|_{1,\infty,K}\le C_f n^{1-m},
\]
where $C_f$ is independent of $n$.
Set $g=f-p^*$. By linearity and polynomial reproduction,
\[
f-\mathcal{L}(f)=g-\mathcal{L}(g).
\]
Therefore,
\begin{equation}\label{C1error2}
|f-\mathcal{L}(f)|_{1,\infty,K}
\le
|g|_{1,\infty,K}
+
|\mathcal{L}(g)|_{1,\infty,K}.
\end{equation}

Let $N=\dim\mathcal{P}_{d,n}$ and $s=M(d+1)$.
Define
\[
\widetilde g
=
\begin{pmatrix}
g|_A\\
\partial_1g|_A\\
\vdots\\
\partial_dg|_A
\end{pmatrix}
\in\mathbb{R}^{s},
\qquad
g|_A=(g(a_1),\ldots,g(a_M))^\top.
\]
Let $Q\in\mathbb{R}^{s\times N}$ be the augmented
evaluation matrix of the discrete orthonormal
Arnoldi basis, so that $Q^\top Q=I_N$. The fitted
function values and first-order derivatives satisfy
\[
\begin{pmatrix}
\mathcal{L}(g)|_A\\
\partial_1\mathcal{L}(g)|_A\\
\vdots\\
\partial_d\mathcal{L}(g)|_A
\end{pmatrix}
=
QQ^\top\widetilde g.
\]

For each $i=1,\ldots,d$,
$\partial_i\mathcal{L}(g)\in
\mathcal{P}_{d,n-1}\subset\mathcal{P}_{d,n}$.
Thus Lemma~\ref{lem:norming} gives
\begin{equation}\label{matrixerror2}
\begin{aligned}
|\mathcal{L}(g)|_{1,\infty,K}
&=
\max_{1\le i\le d}
\|\partial_i\mathcal{L}(g)\|_{\infty,K}\\
&\le
2\max_{1\le i\le d}
\|(\partial_i\mathcal{L}(g))|_A\|_{\ell^\infty}\\
&\le
2\|QQ^\top\widetilde g\|_{\ell^\infty}\\
&\le
2\|QQ^\top\|_\infty
\|\widetilde g\|_{\ell^\infty},
\end{aligned}
\end{equation}
where $\|\cdot\|_\infty$ denotes the induced
maximum-row-sum matrix norm.

To estimate this matrix norm, set $P=QQ^\top$.
Similarly to the Cauchy-Schwarz process in last theorem,
\[
\|QQ^\top\|_\infty
\le
\sqrt{M(d+1)}\,\|QQ^\top\|_2
=
\sqrt{M(d+1)}.
\]

Moreover, the simultaneous approximation estimates
imply, for $n\ge1$,
\[
\begin{aligned}
\|\widetilde g\|_{\ell^\infty}
&=
\max\left\{
\|g|_A\|_{\ell^\infty},
\max_{1\le i\le d}
\|(\partial_i g)|_A\|_{\ell^\infty}
\right\}\\
&\le
\max\left\{
\|g\|_{\infty,K},
|g|_{1,\infty,K}
\right\}\\
&\le C_f n^{1-m}.
\end{aligned}
\]
Substituting into \eqref{matrixerror2} yields
\[
|\mathcal{L}(g)|_{1,\infty,K}
\le
2C_f\sqrt{M(d+1)}\,n^{1-m}.
\]
Combining this estimate with \eqref{C1error2}, we obtain
\[
\begin{aligned}
|f-\mathcal{L}(f)|_{1,\infty,K}
&\le
C_f\bigl(1+2\sqrt{M(d+1)}\bigr)n^{1-m}\\
&\le
3C_f\sqrt{M(d+1)}\,n^{1-m}.
\end{aligned}
\]
This proves the asserted bound, with a constant
independent of $n$, $M$, and the sampling set.
\end{proof}
\subsubsection{Achieve $C^{1}$ estimate without the sampling data of the derivative values.}
\begin{theorem}\label{thm:C1-values}
Let $\Omega\subset\mathbb{R}^d$ be a bounded convex
domain and set $K=\overline{\Omega}$. Then $K$
admits a Markov inequality with a fixed exponent
$r=2$, and let $c_K>0$ be the constant in
Lemma~\ref{lem:norming}.
Let $m\ge1$ be an integer, and suppose that $f \in C^m$ satisfies the extension assumption.

For an integer $n\ge1$, let
$A=\{a_1,\ldots,a_M\}\subset K$ be a set of distinct
sampling points satisfying
\[
h_{A,K}:=
\sup_{x\in K}\min_{a\in A}\|x-a\|_2
\le c_K n^{-2}.
\]
Let $\mathcal{L}_n f\in\mathcal{P}_{d,n}$ be the
discrete least-squares approximation based only on
function values, computed using the Arnoldi basis:
\[
\mathcal{L}_n f
=
\operatorname*{arg\,min}_{p\in\mathcal{P}_{d,n}}
\sum_{j=1}^{M}|f(a_j)-p(a_j)|^2.
\]
Then there exists a constant $C>0$, independent of
$n$, $M$, and the sampling set $A$, such that
\begin{equation}\label{noderi}
|f-\mathcal{L}_n f|_{1,\infty,K}
\le C\sqrt{M}\,n^{2-m},
\end{equation}
where
\[
|u|_{1,\infty,K}
:=
\max_{1\le i\le d}
\|\partial_i u\|_{\infty,K}.
\]
The constant may depend on $f$, $m$, $d$, $K$, and
the fixed Markov exponent and constant.
\end{theorem}

\begin{proof}
We use the first-order seminorm
\[
|u|_{1,\infty,K}
:=
\max_{1\le i\le d}\|\partial_i u\|_{\infty,K}.
\]
By Lemma~\ref{lem:norming}, the polynomial evaluation
matrix has full column rank. Hence $\mathcal{L}_n$
is a well-defined linear operator satisfying
\[
\mathcal{L}_n p=p,
\qquad p\in\mathcal{P}_{d,n}.
\]

Under the extension assumption, a smooth cutoff yields
a compactly supported $C^m$ extension of $f$.
Theorem~\ref{simul} therefore provides a polynomial
$p_n^*\in\mathcal{P}_{d,n}$ such that
\[
\|f-p_n^*\|_{\infty,K}\le C_f n^{-m},
\qquad
|f-p_n^*|_{1,\infty,K}\le C_f n^{1-m},
\]
where $C_f$ is independent of $n$.
Set $g=f-p_n^*$. By linearity and polynomial reproduction,
\[
f-\mathcal{L}_n f=g-\mathcal{L}_n g.
\]
Consequently,
\begin{equation}\label{eq:C1-values-decomposition}
|f-\mathcal{L}_n f|_{1,\infty,K}
\le
|g|_{1,\infty,K}
+
|\mathcal{L}_n g|_{1,\infty,K}.
\end{equation}

Let $N=\dim\mathcal{P}_{d,n}$, and let
$Q\in\mathbb{R}^{M\times N}$ be the evaluation matrix
of the discrete orthonormal Arnoldi basis, so that
$Q^\top Q=I_N$. For the sampling vector
\[
\widetilde g
=
\bigl(g(a_1),\ldots,g(a_M)\bigr)^\top,
\]
the fitted function values satisfy
\[
(\mathcal{L}_n g)|_A=QQ^\top\widetilde g.
\]
By Lemma~\ref{lem:norming},
\[
\begin{aligned}
\|\mathcal{L}_n g\|_{\infty,K}
&\le
2\|(\mathcal{L}_n g)|_A\|_{\ell^\infty}\\
&=
2\|QQ^\top\widetilde g\|_{\ell^\infty}\\
&\le
2\|QQ^\top\|_\infty
\|\widetilde g\|_{\ell^\infty},
\end{aligned}
\]
where $\|\cdot\|_\infty$ denotes the induced
maximum-row-sum matrix norm.

Besides, we still have
\[
\|QQ^\top\|_\infty
\le
\sqrt{M}\,\|QQ^\top\|_2
=
\sqrt{M}.
\]
Together with
\[
\|\widetilde g\|_{\ell^\infty}
\le \|g\|_{\infty,K},
\]
this gives
\[
\|\mathcal{L}_n g\|_{\infty,K}
\le 2\sqrt{M}\,\|g\|_{\infty,K}.
\]

Since $\mathcal{L}_n g\in\mathcal{P}_{d,n}$, the
Markov inequality implies
\[
\begin{aligned}
|\mathcal{L}_n g|_{1,\infty,K}
&\le
\sup_{x\in K}\|\nabla\mathcal{L}_n g(x)\|_2\\
&\le
B_K n^2\|\mathcal{L}_n g\|_{\infty,K}\\
&\le
2B_K n^2\sqrt{M}\,\|g\|_{\infty,K}\\
&\le
2B_K C_f\sqrt{M}\,n^{2-m},
\end{aligned}
\]
where $B_K$ is independent of $n$.
The first inequality follows from
$|\partial_i u(x)|\le\|\nabla u(x)\|_2$ for every $i$.

Substituting into \eqref{eq:C1-values-decomposition}
yields
\[
|f-\mathcal{L}_n f|_{1,\infty,K}
\le
C_f n^{1-m}
+
2B_K C_f\sqrt{M}\,n^{2-m}.
\]
Since $r\ge1$, $n\ge1$, and $M\ge1$, we conclude that
\[
|f-\mathcal{L}_n f|_{1,\infty,K}
\le
C_f(1+2B_K)\sqrt{M}\,n^{2-m}.
\]
This proves \eqref{noderi}, with a constant independent
of $n$, $M$, and the sampling set.
\end{proof}
\begin{remark}\label{rem:net-rates}
The estimates in Theorems~3--5 involve the factors
$\sqrt{M}$ or $\sqrt{M(d+1)}$.
To interpret these estimates as $n\to\infty$, fix
the dimension $d$ and choose sampling sets
$A_n\subset K$ satisfying
\[
h_{A_n,K}\leq c_K n^{-2},
\qquad
M_n:=|A_n|\asymp n^{2d}.
\]
Such a choice is possible, for example, using
sufficiently fine uniform Cartesian grids on a
fixed rectangular domain. The condition on $M_n$
is an additional sampling requirement; it does not
follow from the fill-distance upper bound alone.

In the following estimates, $\mathcal{L}(f)$ denotes
the degree-$n$ least-squares approximation based on
the sampling data specified in each case.
Substituting $M_n\asymp n^{2d}$ into Theorems~3--5
gives
\[
\|f-\mathcal{L}(f)\|_{\infty,K}
=
\mathcal{O}(n^{d-m})
\]
when only function values are sampled,
\[
|f-\mathcal{L}(f)|_{1,\infty,K}
=
\mathcal{O}(n^{d+1-m})
\]
when function values and all first partial
derivatives are sampled, and
\[
|f-\mathcal{L}(f)|_{1,\infty,K}
=
\mathcal{O}(n^{d+2-m})
\]
for the derivative error when only function values
are sampled.
These bounds guarantee convergence when
$m>d$, $m>d+1$, and $m>d+2$, respectively.
Thus, sampling first derivatives improves the
derivative-error bound by one power of $n$.

If $f$ admits a $C^\infty$ extension to a
neighborhood of $K$, these estimates apply for
every fixed integer $m$. Consequently, all three
errors decay faster than any fixed algebraic
power of $n$ along the specified sampling sequence.
The constants may depend on $m$, but are
independent of $n$.

Finally, the factors $\sqrt{M}$ and $\sqrt{M(d+1)}$
arise from the matrix-norm estimate
\[
\|QQ^\top\|_\infty
\leq \sqrt{s}\,\|QQ^\top\|_2
=\sqrt{s},
\]
where $s=M$ for function-value sampling and
$s=M(d+1)$ when first derivatives are also sampled.
These factors are not claimed to be sharp, and
their presence in the bounds does not imply that
the actual error increases with the number of
samples.
\end{remark}

\section{The error propagation by exponential map}
\label{sec5}
As the process of the THI method stated, the exponential map is used to pull the vector on the tangent plane back to the point on the manifold. So it's important to know the error's propagation.

\begin{theorem}
(\textbf{Theorem 3.1 in \cite{JSC2}}) Let $\mathcal{M}$ be a Riemannian manifold, let $q \in \mathcal{M}$ and consider tangent vectors $\Delta, \tilde{\Delta} \in T_q \mathcal{M}$, which are to be interpreted as exact datum and associated approximation. Write $\delta=\|\Delta\|, \tilde{\delta}=\|\tilde{\Delta}\|$, where it is understood that the norm is that of $T_q \mathcal{M}$. Assume that $\delta, \tilde{\delta}<1$. Let $\sigma=\operatorname{span}(\Delta, \tilde{\Delta}) \subset T_q \mathcal{M}$ and let $K_q(\sigma)$ be the sectional curvature at $q$ with respect to the 2 -plane $\sigma$.

If $s_0=\angle(\tilde{\Delta}, \Delta)$ is the angle between $\tilde{\Delta}$ and $\Delta$, then the Riemannian distance between the manifold locations $p=\operatorname{Exp}_q^{\mathcal{M}}(\Delta)$ and $\tilde{p}=\operatorname{Exp}_q^{\mathcal{M}}(\tilde{\Delta})$ is
\begin{equation}\label{error propagation}
\operatorname{dist}_{\mathcal{M}}(p, \tilde{p}) \leq|\delta-\tilde{\delta}|+s_0 \delta\left(1-\frac{K_q(\sigma)}{6} \delta^2+o\left(\delta^2\right)\right)+\mathcal{O}\left(s_0^2\right)
\end{equation}
with the underlying assumption that all data is within the injectivity radius at $q$.
\end{theorem}

By theoretical results established in \cite{JSC2}, when the sectional curvature of the manifold is bounded, then the upper bound (\ref{error propagation}) can be sharped to 
$$
\operatorname{dist}_{\mathcal{M}}(p, \tilde{p})=|\delta-\tilde{\delta}|+\epsilon\left(1-\frac{K_q(\sigma)}{6} \delta^2+o\left(\delta^2\right)\right)+\mathcal{O}\left(\epsilon^2\right)
$$
where $\epsilon:=\|E\|=\|\tilde{\Delta}-\Delta\|$. So the order of $T_q \mathcal{M}$-interpolation error will be the same as the $\mathcal{M}$-interpolation's error order.

Let \(q_0\in\mathcal M\) be fixed, and write
\[
w(x)=\operatorname{Log}_{q_0}(f(x)),
\qquad
\widehat f(x)=\operatorname{Exp}_{q_0}(\widehat w(x)),
\]
where \(\widehat w\) is the polynomial approximation of \(w\). Assume that \(w(K)\) and \(\widehat w(K)\) lie in a fixed compact convex subset \(B\subset T_{q_0}\mathcal M\) contained in the domain of the normal coordinates. All constants below are independent of the polynomial degree and the number of sampling points.

For the \(C^0\) error, Theorem 6 applies pointwise with
\[
\Delta=w(x),\qquad
\widetilde\Delta=\widehat w(x),\qquad
p=f(x),\qquad
\widetilde p=\widehat f(x),
\]
provided its hypotheses hold. It describes how the tangent-space perturbation affects the Riemannian distance, including the effect of sectional curvature. The uniform bound needed here follows from the smoothness of the exponential map on \(B\). Indeed, the curve
\[
\gamma_x(t)=\operatorname{Exp}_{q_0}
\bigl(w(x)+t(\widehat w(x)-w(x))\bigr),
\qquad 0\le t\le1,
\]
joins \(f(x)\) to \(\widehat f(x)\). Since the differential of the exponential map is uniformly bounded on \(B\), its length satisfies
\[
\begin{aligned}
\operatorname{dist}_{\mathcal M}(f(x),\widehat f(x))
&\le \int_0^1\|\gamma_x'(t)\|_{\mathcal M}\,dt\\
&\le C\|\widehat w(x)-w(x)\|.
\end{aligned}
\]
Consequently,
$$
\sup _{x \in K} \operatorname{dist}_{\mathcal{M}}(f(x), \widehat{f}(x)) \leq C\|w-\widehat{w}\|_{\infty, K} .
$$
Thus, the \(C^0\) convergence bound in the tangent space carries over to the manifold.

For the \(C^1\) error, assume additionally that \(\mathcal M\) is smoothly embedded in a Euclidean space, and measure derivative errors in that ambient space. Regard \(\operatorname{Exp}_{q_0}\) as an ambient-space-valued map. By the chain rule,
\[
\partial_i f
=d(\operatorname{Exp}_{q_0})_w[\partial_i w],
\qquad
\partial_i\widehat f
=d(\operatorname{Exp}_{q_0})_{\widehat w}
[\partial_i\widehat w].
\]
Adding and subtracting
\(d(\operatorname{Exp}_{q_0})_{\widehat w}[\partial_i w]\)
gives
\[
\begin{aligned}
\partial_i f-\partial_i\widehat f
={}&d(\operatorname{Exp}_{q_0})_{\widehat w}
[\partial_i w-\partial_i\widehat w]\\
&+\bigl(
d(\operatorname{Exp}_{q_0})_w
-d(\operatorname{Exp}_{q_0})_{\widehat w}
\bigr)[\partial_i w].
\end{aligned}
\]
The first and second derivatives of the ambient representation of the exponential map are uniformly bounded on \(B\). Moreover, convexity of \(B\) and the fundamental theorem of calculus imply that, for any \(\xi\in T_{q_0}\mathcal M\),
\[
\begin{aligned}
&\bigl(
d(\operatorname{Exp}_{q_0})_w
-d(\operatorname{Exp}_{q_0})_{\widehat w}
\bigr)[\xi]\\
&\qquad=
\int_0^1
d^2(\operatorname{Exp}_{q_0})_{\widehat w+t(w-\widehat w)}
[w-\widehat w,\xi]\,dt.
\end{aligned}
\]
Therefore,
\[
\|\partial_i f-\partial_i\widehat f\|
\le
C\|\partial_i w-\partial_i\widehat w\|
+C\|w-\widehat w\|\,\|\partial_i w\|.
\]
Using the convention
\[
|u|_{1,\infty,K}
=\max_{1\le i\le d}\sup_{x\in K}\|\partial_i u(x)\|,
\]
we obtain
\[
|f-\widehat f|_{1,\infty,K}
\le C\left(
|w-\widehat w|_{1,\infty,K}
+\|w-\widehat w\|_{\infty,K}|w|_{1,\infty,K}
\right).
\]
Since \(w\) is fixed and \(C^1\) on \(K\), its derivative bound can be absorbed into the constant:
$$
|f-\widehat{f}|_{1, \infty, K} \leq C\left(|w-\widehat{w}|_{1, \infty, K}+\|w-\widehat{w}\|_{\infty, K}\right) .
$$

Applying the preceding polynomial approximation estimates componentwise, when only function values are sampled, we obtain
\[
\sup_{x\in K}\operatorname{dist}_{\mathcal M}(f(x),\widehat f(x))
\le C\sqrt M\,n^{-m},
\qquad
|f-\widehat f|_{1,\infty,K}
\le C\sqrt M\,n^{2-m}.
\]
When function values and first derivatives are sampled, the augmented least-squares estimate controls both the function-value and first-derivative errors of \(w\) by \(C\sqrt{M(d+1)}\,n^{1-m}\). Hence
\[
\sup_{x\in K}\operatorname{dist}_{\mathcal M}(f(x),\widehat f(x))
+|f-\widehat f|_{1,\infty,K}
\le C\sqrt{M(d+1)}\,n^{1-m}.
\]
These bounds hold under the assumptions of the corresponding approximation theorems. In particular, the distinction between the derivative-error bounds with and without derivative samples is preserved under the exponential map.

\section{Numerical example}
\label{sec6}

In this section, we conduct numerical experiments on two classic Riemannian manifolds: the unit sphere $S^2$ and the rotation group $\mathrm{SO}(3)$. The code could be found in GitHub\footnote{\url{https://github.com/YuxuanLi18/Hermite-interpolation-of-manifold-data}}.
 
For each manifold we consider two sampling strategies: a uniform Cartesian grid and a Chebyshev grid. The uniform grids are combined with the original smooth test functions, while the Chebyshev grids, which are naturally suited for resolving high‑frequency oscillations, are paired with more oscillatory test functions to examine the behaviour of the methods under increased complexity. 

On all test cases we compare the classical tangent‑space Hermite interpolation (THI+Kriging) with the proposed $G$-Arnoldi enhanced approach, both with and without derivative data, and report function value errors, derivative accuracy (by finite difference), offline and online time.

Here, \textbf{offline time} refers to the one‑time cost of constructing the interpolant from the training data (e.g., building the GEK system, performing Arnoldi orthogonalisation, and computing coefficients). \textbf{Online time per query} is the average cost of evaluating the precomputed interpolant at a single new point, which includes basis evaluation, linear combination, and the exponential map on \(SO(3)\). The offline phase is amortised over many queries, while the online time determines real‑time prediction efficiency.

The sampling sizes and polynomial degrees balance approximation capacity and computational cost. On a \(k\times k\) tensor-product grid, function values uniquely determine polynomials in \(\mathcal P_{2,n}\) for \(n\le k-1\), while including both first partial derivatives extends this range to \(n\le 2k-1\). The degrees used lie within these ranges, ensuring full column rank of the corresponding evaluation matrices. These conditions are distinct from the stronger fill-distance condition in Section 4. The main comparisons are presented in Tables 1, 3, 5, and 7, where THI+Kriging and the Arnoldi-enhanced method use the same sampling points and first-derivative data. Tables 2, 4, 6, and 8 serve as supplementary examples demonstrating approximation of both function values and derivatives without derivative samples, rather than comparisons at a matched polynomial degree.

\subsection{SO3 manifold}
\subsubsection{Uniform sample plan.}
We consider the multivariate Hermite interpolation problem on the special orthogonal group \(SO(3)\). The test function is 
\[
f(\omega) = \exp_m\!\bigl( X(\omega) \bigr) \in SO(3),\qquad
X(\omega) = \begin{pmatrix}
0 & \omega_1 & \omega_2 \\
-\omega_1 & 0 & \omega_1\omega_2 \\
-\omega_2 & -\omega_1\omega_2 & 0
\end{pmatrix}
\]

Here $\omega = (\omega_1,\omega_2) \in [-0.5,0.5]^2$ and \(\exp_m\) denotes the matrix exponential. The sample locations $P_j=\exp _m X\left(\omega^j\right)$ at $\omega^j=\left(\omega_1^j, \omega_2^j\right)$ and the corresponding partial derivatives $V_j^i=\left.\frac{\mathrm{d}}{\mathrm{d} t}\right|_{t=0} \exp _m\left(X\left(\omega^j+t e_i\right)\right)=\mathrm{d}\left(\exp _m\right)\left(X\left(\omega^j\right)\right)\left[\partial_i X\left(\omega^j\right)\right], i=1,2$ of the test function can be obtained by Mathias' theorem (Theorem~3.6 in \cite{higham}.)
$$
\exp _m\left(\begin{array}{cc}
X\left(\omega^j\right) & \partial_i X\left(\omega^j\right) \\
0 & X\left(\omega^j\right)
\end{array}\right)=\left(\begin{array}{cc}
\exp _m\left(X\left(\omega^j\right)\right) & \mathrm{d}\left(\exp _m\right)\left(X\left(\omega^j\right)\right)\left[\partial_i X\left(\omega^j\right)\right] \\
0 & \exp _m\left(X\left(\omega^j\right)\right)
\end{array}\right) .
$$

We evaluate the method on a uniform Cartesian grid of \(7\times 7\) points over the domain \([-0.5,0.5]^2\), which produces \(k=49\) sample locations \(\omega_j\). At each sample we evaluate the function value \(P_j = f(\omega_j)\) and the two partial derivatives \(V_j^i = \partial_i f(\omega_j)\in T_{P_j}SO(3)\). The numerical results are reported in Tables~\ref{tab:so3} and~\ref{tab:mvga-noder}. On the \(7\times 7\) grid, the Arnoldi-based Hermite solver (MV+G-A with full derivative data) reduces the average error from \(5.2\times10^{-7}\) (for THI+Kriging) to \(1.7\times10^{-12}\), while cutting the online evaluation time by two orders of magnitude—from \(5.3\times10^{-2}\)~s down to \(4.2\times10^{-4}\)~s. Even when only function values are supplied, the Arnoldi method still achieves an average error of \(4.0\times10^{-12}\), which exhibits the advantage of the \(G\)-orthogonal polynomial basis.

\begin{table}[!ht]
\centering
\caption{THI (Kriging) vs THI+Arnoldi (with derivative data) on SO(3) with 49 sampling points and polynomial total degree 6.}
\label{tab:so3}
\begin{tabular}{l S[table-format=1.4e-1] S[table-format=1.4e-1]}
\toprule
Metric & {THI (Kriging)} & {THI+Arnoldi} \\
\midrule
Offline time (s)              & 4.2395e+00 & 5.1288e-01 \\
Online time (s per query)     & 5.3090e-02 & 4.1542e-04 \\
Max error                     & 6.7179e-06 & 4.6218e-12 \\
Avg error                     & 5.2085e-07 & 1.7312e-12 \\
FD error d1f                  & 5.8274e-05 & 1.5427e-11 \\
FD error d2f                  & 5.6566e-05 & 2.7587e-11 \\
\bottomrule
\end{tabular}
\end{table}

\begin{table}[h]
\centering
\caption{Performance of MV+G-A without derivative data (function values only) with 49 sampling points and polynomial total degree 6.}
\label{tab:mvga-noder}
\begin{tabular}{l r}
\toprule
Metric & Value \\
\midrule
Offline time (s)              & 1.6558e+0 \\
Online time per query (s)     & 8.7380e-4 \\
Max error (function)          & 1.5088e-11 \\
Avg error (function)          & 4.0359e-12 \\
FD error $\partial_1 f$       & 1.1280e-10 \\
FD error $\partial_2 f$       & 1.1810e-10 \\
\bottomrule
\end{tabular}
\end{table}

\subsubsection{Chebyshev sampling plan.}
For test function which is more complicated, we try
$$
\begin{aligned}
&f:[a, b]^2 \rightarrow S O(3), \quad\left(\omega_1, \omega_2\right) \mapsto \exp _m X\left(\omega_1, \omega_2\right), \text { where }\\
&X\left(\omega_1, \omega_2\right)=\left(\begin{array}{ccc}
0 & \omega_1^2+\frac{1}{2} \omega_2 & \sin \left(4 \pi\left(\omega_1^2+\omega_2^2\right)\right) \\
-\omega_1^2-\frac{1}{2} \omega_2 & 0 & \omega_1+\omega_2^2 \\
-\sin \left(4 \pi\left(\omega_1^2+\omega_2^2\right)\right) & -\omega_1-\omega_2^2 & 0
\end{array}\right) .
\end{aligned}
$$
We build Chebyshev sample point
$$
\left\{\left.\frac{1}{2}(b-a) \cos \left(\frac{(2 j-1) \pi}{2 k}\right)+\frac{1}{2}(b+a) \right\rvert\, j=1, \ldots, k\right\}^2 \subset \mathbb{R}^2. 
$$ 

\begin{table}[!ht]
\centering
\caption{THI (Kriging) vs THI+Arnoldi (with derivative data) on SO(3) with 100 sampling points and polynomial total degree 19.}
\label{tab:thi_vs_full_100}
\begin{tabular}{l S[table-format=3.4e2] S[table-format=3.4e2]}
\toprule
Metric                     & {THI(Kriging)}          & {THI+Arnoldi} \\
\midrule
Offline time (s)           & 1.2291e+02     & 1.2146e+00    \\
Online time (s per query)  & 1.3148e-01     & 3.3413e-03    \\
Max error                  & 6.5451e-03     & 1.8523e-04    \\
Avg error                  & 1.7639e-03     & 4.5319e-05    \\
FD error d1f               & 4.0345e-02     & 1.9274e-04    \\
FD error d2f               & 3.8792e-02     & 1.8797e-04    \\
\bottomrule
\end{tabular}
\end{table}

\begin{table}[!ht]
\centering
\caption{Performance of MV+G-A without derivative data (function values only) with 225 sampling points and polynomial total degree 14.}
\label{tab:mvga_noder}
\begin{tabular}{l S[table-format=1.4e-1]}
\toprule
Metric & {Value} \\
\midrule
Offline time (s)              & 2.6951e+00 \\
Online time (s per query)     & 2.7387e-03 \\
Max error                     & 1.7172e-03 \\
Avg error                     & 3.7499e-04 \\
FD error $\partial_1 f$       & 1.7103e-02 \\
FD error $\partial_2 f$       & 1.7087e-02 \\
\bottomrule
\end{tabular}
\end{table}

Under a more challenging test function with a finer Chebyshev grid (Tables~\ref{tab:thi_vs_full_100} and~\ref{tab:mvga_noder}), the full MV+G-A method maintains an average error of $4.5\times10^{-5}$ and an online query time of $3.3\times10^{-3}$~seconds. In contrast, THI+Kriging degrades to an average error of $1.8\times10^{-3}$ and a query time of $1.3\times10^{-1}$~seconds. On a $15\times15$ grid, the derivative-free variant achieves an average error of $3.7\times10^{-4}$, with finite-difference derivative errors around $10^{-2}$. This suggests that even without direct derivative samples, the polynomial model still captures derivative information reasonably well.

\subsection{Sphere manifold}
\subsubsection{Uniform sampling plan.}
We consider the multivariate Hermite interpolation problem on the unit sphere \(\mathcal{M}=S^2\subset\mathbb{R}^3\).

The Riemannian exponential and logarithmic maps on $S^2$ are
$$
\begin{aligned}
\operatorname{Exp}_q(v) & =\cos \left(\|v\|_2\right) q+\sin \left(\|v\|_2\right) \frac{v}{\|v\|_2} \in S^2 \\
\log _q(p) & =\arccos (\langle q, p\rangle) \frac{p-\langle q, p\rangle q}{\|p-\langle q, p\rangle q\|_2} \in T_q S^2
\end{aligned}
$$
respectively. As a test function, we use the Gauß map of the helicoid in $\mathbb{R}^3$,
$$
f:[a, b]^2 \rightarrow S^2, \quad\left(\omega_1, \omega_2\right) \mapsto \frac{1}{\left(\exp({2 \omega_1})+1\right)}\left(\begin{array}{c}
2 \exp({ \omega_1}) \cos \left(\omega_2\right) \\
2 \exp({ \omega_1}) \sin \left(\omega_2\right) \\
\exp({2 \omega_1})-1
\end{array}\right)
$$
The function produces the normal field of the helicoid and is obtained from the stereographic projection of $\left(\omega_1, \omega_2\right) \mapsto \exp \left(\omega_1+i \omega_2\right) \in \mathbb{C}$ onto $S^2$.
The partial derivatives of $f$ are
$$
\begin{aligned}
& \partial_1 f\left(\omega_1, \omega_2\right)=\frac{-2 \exp({2 \omega_1})}{\left(\exp({2 \omega_1})+1\right)^2}\left(\begin{array}{c}
2 \exp({\omega_1}) \cos \left(\omega_2\right) \\
2 \exp({ \omega_1}) \sin \left(\omega_2\right) \\
\exp({2 \omega_1})-1
\end{array}\right)+\frac{2}{\exp({2 \omega_1})+1}\left(\begin{array}{c}
\exp({\omega_1}) \cos \left(\omega_2\right) \\
\exp({\omega_1}) \sin \left(\omega_2\right) \\
\exp({2 \omega_1})
\end{array}\right) \\
& \partial_2 f\left(\omega_1, \omega_2\right)=\frac{1}{\exp({2 \omega_1})+1}\left(\begin{array}{c}
-2 \exp({\omega_1}) \sin \left(\omega_2\right) \\
2 \exp({\omega_1}) \cos \left(\omega_2\right) \\
0
\end{array}\right)
\end{aligned}
$$

We use a uniform Cartesian grid of \(8\times 8\) points in the parameter domain \(\Omega\) with $\left(\omega_1^j, \omega_2^j\right) \in[a, b]^2=\left[-\frac{1}{4} \pi, \frac{1}{4} \pi\right]^2$, resulting in $k = 64$ sample locations $\omega_j$. At each sample we evaluate the function value $p_j = f(\omega_j)\in S^2$ and the two partial derivatives $v_j^i = \partial_i f(\omega_j)\in T_{p_j}S^2$ ($i=1,2$). Table~\ref{tab:sphere} and Table~\ref{tab:mvga-noder-sphere} display the performance of THI and THI+Arnoldi for the cases with and without derivative data, respectively. We observe that considerably better performance is achieved with the $G$-Arnoldi orthogonalization.

\begin{table}[!ht]
\centering
\caption{THI (Kriging) vs THI+Arnoldi on sphere with 64 sampling points and polynomial total degree 15.}
\label{tab:sphere}
\begin{tabular}{l S[table-format=1.4e-1] S[table-format=1.4e-1]}
\toprule
Metric & {THI (Kriging)} & {THI+Arnoldi} \\
\midrule
Offline time (s)              & 1.2368e+01 & 6.2331e+00 \\
Online time (s per query)     & 9.7652e-02 & 5.9104e-01 \\
Avg error                     & 5.7095e-06 & 4.6558e-10 \\
Max error                     & 2.0834e-05 & 3.4082e-09 \\
FD error d1f                  & 1.3808e-04 & 7.8206e-10 \\
FD error d2f                  & 1.3269e-04 & 8.9595e-10 \\
\bottomrule
\end{tabular}
\end{table}

\begin{table}[h]
\centering
\caption{Performance of MV+G-A without derivative data (sphere S$^2$) by 64 sampling points and polynomial total degree 7.}
\label{tab:mvga-noder-sphere}
\begin{tabular}{l r}
\toprule
Metric & Value \\
\midrule
Offline time (s)              & $2.6230 \times 10^{+1}$ \\
Online time per query (s)     & $2.0807 \times 10^{-1}$ \\
Avg error (function)          & $7.0428 \times 10^{-6}$ \\
Max error (function)          & $2.4243 \times 10^{-5}$ \\
FD error $\partial_1 f$       & $1.9167 \times 10^{-4}$ \\
FD error $\partial_2 f$       & $5.3635 \times 10^{-5}$ \\
\bottomrule
\end{tabular}
\end{table}

\subsubsection{Chebyshev sampling plan.}
As a slightly more complicated test function, we modify the helicoid Gauß map by doubling the frequency in the second variable:
\[
f:[a,b]^2\to S^2,\quad (\omega_1,\omega_2)\mapsto
\frac{1}{\exp({2 \omega_1})+1}
\begin{pmatrix}
2\exp({ \omega_1})\cos(2\omega_2)\\[2pt]
2\exp({ \omega_1})\sin(2\omega_2)\\[2pt]
\exp({2 \omega_1})-1
\end{pmatrix}.
\]
The partial derivatives are given by
\[
\begin{aligned}
\partial_1 f(\omega_1,\omega_2)&=
\frac{-2\exp({2 \omega_1})}{(\exp({2 \omega_1})+1)^2}
\begin{pmatrix}
2\exp({ \omega_1})\cos(2\omega_2)\\
2\exp({ \omega_1})\sin(2\omega_2)\\
\exp({2 \omega_1})-1
\end{pmatrix}
+\frac{2}{\exp({2 \omega_1})+1}
\begin{pmatrix}
\exp({ \omega_1})\cos(2\omega_2)\\
\exp({ \omega_1})\sin(2\omega_2)\\
\exp({2 \omega_1})
\end{pmatrix},\\[6pt]
\partial_2 f(\omega_1,\omega_2)&=
\frac{1}{\exp({2 \omega_1})+1}
\begin{pmatrix}
-4\exp({ \omega_1})\sin(2\omega_2)\\
4\exp({ \omega_1})\cos(2\omega_2)\\
0
\end{pmatrix}.
\end{aligned}
\]

We sample the parameter domain $[-\frac{1}{4} \pi, \frac{1}{4} \pi]^2$ using a Chebyshev grid.  
For a given number of points $N$ per dimension, the one-dimensional nodes are
\[
x_j = \frac{a+b}{2} + \frac{b-a}{2}\cos\!\left(\frac{j\pi}{N+1}\right),\qquad j=1,\dots,N.
\]
The two-dimensional grid is then formed by the Cartesian product of these nodes.  
We take $N=10$ in the experiment, and it gives $k=100$ sample locations $\omega_j$.

\begin{table}[!ht]
\centering
\caption{THI (Kriging) vs THI+Arnoldi on the sphere with $100$ Chebyshev sampling points and polynomial total degree $15$.}
\label{tab:sphere_thi_mvga_full}
\begin{tabular}{l S[table-format=1.4e-1] S[table-format=1.4e-1]}
\toprule
Metric & {THI (Kriging)} & {THI+Arnoldi} \\
\midrule
Offline time (s)              & 5.5417e+01 & 7.0162e-02 \\
Online time (s per query)     & 2.5090e-01 & 3.1461e-03 \\
Avg error                     & 2.5286e-04 & 8.9908e-06 \\
Max error                     & 4.3919e-03 & 1.0069e-04 \\
FD error $\partial_1 f$       & 8.3956e-03 & 5.2096e-05 \\
FD error $\partial_2 f$       & 7.8004e-03 & 5.7229e-05 \\
\bottomrule
\end{tabular}
\end{table}

\begin{table}[!ht]
\centering
\caption{Performance of MV+G-A without derivative data by 100 Chebyshev sampling points and polynomial total degree 9.}
\label{tab:mvga_noder_sphere}
\begin{tabular}{l r}
\toprule
Metric & Value \\
\midrule
Offline time (s)              & $2.8171\times 10^{-2}$ \\
Online time per query (s)     & $1.1330\times 10^{-3}$ \\
Avg error (function)          & $3.3172\times 10^{-4}$ \\
Max error (function)          & $3.7130\times 10^{-3}$ \\
FD error $\partial_1 f$       & $7.3057\times 10^{-3}$ \\
FD error $\partial_2 f$       & $2.3538\times 10^{-3}$ \\
\bottomrule
\end{tabular}
\end{table}

Similar to the uniform grid case, Table~\ref{tab:sphere_thi_mvga_full} and Table~\ref{tab:mvga_noder_sphere} show that the Arnoldi-based Hermite solver (MV+G-A with full derivative data) consistently outperforms classical THI+Kriging in both accuracy and speed on Chebyshev sampling points. On a $10\times 10$ Chebyshev grid, the Arnoldi method achieves an average error of $8.9\times10^{-6}$ and $3.14\times10^{-3}$ seconds per query, compared to $2.5\times10^{-4}$ error and $2.5\times10^{-1}$ seconds for THI+Kriging. Using only function values (MV+G-A without derivative data) also produces reasonable accuracy, with derivative errors roughly one order of magnitude larger, which show the benefit of Hermite data. In summary, the $G$-Arnoldi basis delivers a well-conditioned least-squares problem with fast online evaluation and high accuracy.

\section{Conclusion}

We have introduced an Arnoldi-enhanced tangent space Hermite interpolation (THI) framework for stable and efficient approximation of manifold-valued data. The multivariate $G$-Arnoldi process replaces the traditional Kriging-based interpolant with a discrete orthogonal polynomial basis, avoiding the exponential growth of the condition number even with scattered derivative data.

We established $C^0$ and $C^1$ error bounds for the multivariate Arnoldi approximation. The analysis shows how the total degree of the polynomial basis and the sampling density interact to maintain convergence rates.
Numerical experiments on $SO(3)$ and $S^2$ show that the proposed method gives better accuracy. Compared to Kriging, it substantially reduces the offline construction cost and the online query latency in most test cases.

As future work, the idea can be extended to approximation problems on Grassmann and Stiefel manifolds, and to combine adaptive sampling strategies with the $G$-Arnoldi process for functions exhibiting sharp local variations. 
\backmatter

\section*{Declaration of competing interests}
The authors declare no competing interests.

\section*{Funding sources}
The research is partially supported by Research Development Fund at XJTLU with Nos. RDF-TP-0014 and RDS10120260083.

\section*{Author Contributions}  
 Yuxuan Li: Methodology, Software, Writing – original draft; Qiang Niu: Conceptualization, Supervision, Writing – review \& editing; Wubin Zhou: Supervision, Writing – review \& editing.

\section*{Acknowledgement}
We thank Prof. Jie Fei and Prof. Haixia Liang for many useful discussions of the work.

\section*{Data availability}  The code supporting this study is openly available at \url{https://github.com/YuxuanLi18/Hermite-interpolation-of-manifold-data}. 

\section*{Declaration of generative AI use}

During the preparation of this work the authors used DeepSeek in order to improve the readability and language of the manuscript. The authors also used ChatGPT for assistance with some technical details. After using these tools, the authors reviewed and edited the content as needed and take full responsibility for the content of the published article.

\begin{appendices}

\end{appendices}

\bibliographystyle{plain}

\end{document}